\documentclass[reqno]{amsart}
\usepackage{amsmath}
\usepackage{amsfonts}%
\usepackage{amssymb}%
\usepackage{graphicx}
\theoremstyle{plain}

\newtheorem{lemma}{Lemma}

\numberwithin{equation}{section}
\begin{document}
\title[A note on the complex rewriting of the elliptic first order systems in the plane]{A note on elliptic first order systems in the plane and the Vekua Equation with structure polynomial $X^2 + \beta X + \alpha$}
\begin{abstract}
This note deals with the following problem: under which conditions can elliptic first order systems in the plane be complex-rewritten as a parameter-depending Vekua-type equation? 
\end{abstract}
\author{D. Alay\'on-Solarz }
\maketitle
\section{Introduction}

We consider the Vekua equation
\begin{displaymath}
\partial_{\bar z}  w + A \cdot w + B \cdot \overline{w} = F
\end{displaymath}
where $\partial_{\bar z}$ denotes the Cauchy-Riemann operators and $A,B$ and $F$ are measurable functions. This class of functions is very general: any differentiable complex function $w$ without zeros is a solution for this equation, for this we choose $A,B$ and $F$ as:
\begin{displaymath}
A = -\frac{\partial_{\bar z} w}{w}, \ \ \ \ B = 0 \ \ \ \  F = 0.
\end{displaymath}
We recall also that every Pseudoanalytic function in the sense of Bers is a Generalized Analytic Function.
\\
The relationship between Generalized Analytic Functions and first order elliptic systems in the plane is well known: let us consider the real system
\begin{equation}
\left\{\begin{array}{cc}-v_y + a_{11} u_x+ a_{12} u_y+ a_{1} u + b_{1}v = f_1 \\ \  \ v_x + a_{21} u_x+ a_{22} u_y+ a_{2} u + b_{2}v  = f_2\end{array}\right.
\end{equation}
where 
\begin{displaymath}
a_{11}, a_{22} > 0, \ \ \ \ a_{11} a_{22} - \frac{1}{4}(a_{12} + a_{21})^2 > 0.
\end{displaymath}
In this case the system is said to be of elliptic type and every elliptic system can be reduced to the previous form. A system is said to be in the canonical form if
\begin{equation}
\left\{\begin{array}{cc}-  V_y +  U_x+ a U + b V = f\\  \\ \  \  V_x + U_y+ c U + d V  = g\end{array}\right.
\end{equation}
then the real system can be complex rewritten as
\begin{displaymath}
\partial_{\bar z}  w + A \cdot w + B \cdot \overline{w} = 0
\end{displaymath}
where $w$ depends on $u$ and $v$ and the pair $u,v$ is a solution to the real system. The advantage of the latter form is that complex methods can be used to study the corresponding solutions.
\\
\\
In general, the system is assumed uniformly elliptic and by solving an auxiliary Beltrami equation the system is taken to the canonical form, for details we refer to the second chapter of \cite{Vek}. Observe that If the quofficients of the system satisfy the conditions:

\begin{displaymath}
a_{11}= a_{22}, \ \ \ a_{21} = -a_{12}
\end{displaymath} 

then it is not necessary to solve a Beltrami equation, it suffices to apply a substitution given by
\begin{displaymath}
U = a_{11}u
\end{displaymath}
\begin{displaymath}
V = v + a_{12}V
\end{displaymath}
The resulting system is on the canonical form. 
\\
\\
We now include in this note the language of parameter-depending structure polynomial $X^2 + \beta X + \alpha$ where $\alpha$ and $\beta$ are functions depending on the variables $(x,y)$. Here $i$ satisfy 
\begin{displaymath}
i^2 = - \beta i - \alpha
\end{displaymath}
The corresponding algebra of functions is given by
\begin{displaymath}
(u_1 + i v_1) \cdot (u_2 + i v_2) =  (u_1 u_2 - \alpha  v_1 v_2) + i (u_1 v_2 + u_2 v_1 - \beta  v_1 v_2)
\end{displaymath}
The Cauchy-Riemann operator is
\begin{displaymath}
\frac{1}{2} ( \partial_x + i \partial_y)
\end{displaymath}
\
If $w = u + i v$. The equation
\begin{displaymath}
\partial_{\bar z} w = 0
\end{displaymath}
is equivalent to
\begin{displaymath}
 \partial_x u - \alpha \partial_y v = 0
\end{displaymath}
\begin{displaymath}
\partial_y u + \partial_x v - \beta \partial_y v = 0.
\end{displaymath}
Which is a generalization of the ordinary Cauchy-Riemann equations and in the most general case a system of equations with variable coefficients. We will consider the Vekua equation
\begin{displaymath}
\partial_{\overline{z}} W + A \cdot W + B \cdot \overline{W} = F.
\end{displaymath}
but this time interpreted with the structure polynomial $X^2 + \beta X + \alpha$. Naturally, if $\alpha = 1$ and $\beta = 0$ one recovers the ordinary Vekua equation.
\section{The complex rewriting}
With the inclusion of the language of the structure polynomial $X^2+ \beta X + \alpha$ it turns out that \textit{every} first order elliptic system in the plane can be rewritten as Vekua equation by means of a substitution without assuming uniform ellipticity or solving an auxiliary Beltrami equation.
\begin{lemma}
Suppose the main coefficients in the  elliptic first order system (1.1) are differentiable, then the system can be complex rewritten, by means of a substitution, as a Generalized Analytic Function with structure polynomial $X^2 + \beta X + \alpha$, where
\begin{displaymath}
\alpha = \dfrac{a_{22}}{a_{11}} > 0, \ \ \ \  \beta = -\dfrac{a_{21}+a_{12}}{a_{11}}, \ \ a_{11} > 0
\end{displaymath}
\end{lemma}
\begin{proof}
We start by subjecting the system to the substitution given by
\begin{displaymath}
U = a_{22} u, \ \ \ \  V =  v - a_{12}u, \ \ \ \ a_{22}>0,
\end{displaymath}
and we get 
\begin{equation}
\left\{\begin{array}{cc}-V_y + \dfrac{a_{11}}{a_{22}} U_x+ a_{\ast} U + b_{\ast}V = f_1\\  \\ \  \  V_x +\dfrac{( a_{21} + a_{12})}{a_{22}} U_x+ U_y+ c_{\ast} U + d_{\ast}V  = f_2\end{array}\right.
\end{equation}
Now setting
\begin{displaymath}
\alpha = \dfrac{a_{22}}{a_{11}} > 0, \ \ \ \  \beta = -\dfrac{a_{21}+a_{12}}{a_{11}}, \ \ a_{11} > 0
\end{displaymath}
the system is rewritten as
\begin{equation}
\left\{\begin{array}{cc}-V_y + \dfrac{1}{\alpha} U_x+ a_{\ast} U + b_{\ast}V = f_1\\  \\ \  \  V_x -\dfrac{\beta}{\alpha} U_x+ U_y+ c_{\ast} U + d_{\ast}V  = f_2\end{array}\right.
\end{equation}
Now we multiply the first line by $\beta$ and we sum it to the second line and we get
\begin{equation}
\left\{\begin{array}{cc}-V_y + \dfrac{1}{\alpha} U_x+ a_{\ast} U + b_{\ast}V = f_1\\  \\ \  \  V_x -\beta V_y+ U_y+ c U + d V  = g\end{array}\right.
\end{equation}
Finally, we multiply the first line by $\alpha > 0$ to obtain
\begin{equation}
\left\{\begin{array}{cc}- \alpha V_y +  U_x+ a U + b V = f\\  \\ \  \  V_x -\beta V_y+ U_y+ c U + d V  = g\end{array}\right.
\end{equation}
Consider the complex algebra with differentiable structure polynomial $X^2 + \beta X + \alpha$, where $\alpha$ and $\beta$ are the same that were obtained from the elliptic system. Now, following Vekua we introduce the complex map
\begin{displaymath}
W = U + i V,
\end{displaymath}
and
\begin{displaymath}
\overline{W} = U - i V,
\end{displaymath}
and so we may write any elliptic system with the said conditions in the form
\begin{displaymath}
\partial_{\overline{z}} W + A \cdot W + B \cdot \overline{W} = F.
\end{displaymath}
Where
\begin{displaymath}
\partial_{\overline{z}}W = \frac{1}{2}(W_x + i W_y),
\end{displaymath}
\begin{displaymath}
A = \frac{1}{4} \Big( a - \frac{\beta}{\alpha} b + d + i  c - i \frac{1}{\alpha} b  \Big) 
\end{displaymath}
\begin{displaymath}
 B = \frac{1}{4}(a + \frac{\beta}{\alpha} b - d + i c +  i \frac{1}{\alpha} b \Big)
 \end{displaymath}
 \begin{displaymath}
 F = f + i g
 \end{displaymath}
\end{proof}
\section{Concluding remarks on the elliptic case}
Recall that a system (1.1)  is said to be elliptic if
\begin{displaymath}
a_{11} > 0, \ \ a_{22} > 0, 
\end{displaymath}
\begin{displaymath}
\Delta = a_{11} a_{22} - \frac{1}{4}( a_{12} + a_{21})^2 > 0
\end{displaymath}
dividing the last equation by $a^2_{11} > 0$ we get
\begin{displaymath}
\frac{a_{22}}{ a_{11}} - \frac{1}{4}(\frac{ a_{12} + a_{21}}{a_{11}})^2 > 0
\end{displaymath}
which is equivalent to
\begin{displaymath}
4 \alpha - \beta^2 >0
\end{displaymath}
When $\alpha$ and $\beta$ are constants the underlying algebra is the elliptic complex numbers. While all elliptic systems with constant $a_{ij}$, $i,j = 1,2$ coefficients are of this form there is a family of systems with variable coefficients which result in $\alpha$ and $\beta$ being constants.

\section{Acknowledgments}
I am grateful to W. Tutschke and C.J. Vanegas for some discussions.

\end{document}